\documentclass[11pt]{article}

\usepackage{amsfonts,amsmath,latexsym,color,epsfig,enumitem, amssymb,bbm,amsthm}
\usepackage{hyperref}
\usepackage{subcaption}
\usepackage{ulem}

\newtheorem{conj}{Conjecture}

\newtheorem{theorem}{Theorem}[section]

\newtheorem{proposition}[theorem]{Proposition}
\newtheorem{lemma}{Lemma}[section]

\newtheorem{obs}[lemma]{Observation}

\theoremstyle{definition}

\newtheorem{remark}[lemma]{Remark}

\numberwithin{equation}{section}

\renewcommand{\le}{\leqslant}
\renewcommand{\ge}{\geqslant}
\renewcommand{\leq}{\leqslant}
\renewcommand{\geq}{\geqslant}

\newcommand{\E}{\mathbb E}

\newcommand{\LL}{\mathcal L }

\newcommand{\NN}{\mathcal N}

\DeclareMathOperator{\cdim}{codim}

\def\qed{\ifvmode\mbox{ }\else\unskip\fi\hskip 1em plus 10fill$\Box$}

\input{epsf}

\makeatletter
\def\Ddots{\mathinner{\mkern1mu\raise\p@
\vbox{\kern7\p@\hbox{.}}\mkern2mu
\raise4\p@\hbox{.}\mkern2mu\raise7\p@\hbox{.}\mkern1mu}}
\makeatother

\renewcommand{\P}{\mathbb{P}}

\def\Z{\mathbb Z}

\def\F{\mathbb F}
\def\P{\mathbb P}

\def\E{\mathbb E}

\title{\vspace{-0.7cm} On off-diagonal Ramsey numbers for vector spaces over $\mathbb{F}_{2}$}
\author{Zach Hunter\thanks{Department of Mathematics, ETH, Z\"urich, Switzerland. Email: {\tt zach.hunter@math.ethz.ch}.} \and Cosmin Pohoata\thanks{Department of Mathematics, Emory University, Atlanta, GA. Email: {\tt cosmin.pohoata@emory.edu}. Research supported by NSF Award DMS-2246659.}}
\date{}

\begin{document}
\maketitle

\begin{abstract}
\smallskip
For every positive integer $d$, we show that there must exist an absolute constant $c > 0$ such that the following holds: for any integer $n \geq cd^{7}$ and any red-blue coloring of the one-dimensional subspaces of $\mathbb{F}_{2}^{n}$, there must exist either a $d$-dimensional subspace for which all of its one-dimensional subspaces get colored red or a $2$-dimensional subspace for which all of its one-dimensional subspaces get colored blue. This answers recent questions of Nelson and Nomoto, and confirms that for any even plane binary matroid $N$, the class of $N$-free, claw-free binary matroids is polynomially $\chi$-bounded. 

Our argument will proceed via a reduction to a well-studied additive combinatorics problem, originally posed by Green: given a set $A \subset \mathbb{F}_{2}^{n}$ with density $\alpha \in [0,1]$, what is the largest subspace that we can find in $A+A$? Our main contribution to the story is a new result for this problem in the regime where $1/\alpha$ is large with respect to $n$, which utilizes ideas from the recent breakthrough paper of Kelley and Meka on sets of integers without three-term arithmetic progressions. 
\end{abstract}

\section{Introduction}

For every two positive integers $s$ and $d$, let $R_{\mathbb{F}_{2}}(s,d)$ denote the smallest integer $n$ such that in every red-blue coloring of all the one-dimensional subspaces of $\mathbb{F}_{2}^{n}$ there must exist either a $d$-dimensional subspace for which all of its one-dimensional subspaces get colored red or an $s$-dimensional subspace for which all of one-dimensional subspaces get colored blue. This quantity was first introduced in 1972 (in much greater generality) by Graham, Leeb and Rothschild \cite{GLR72}, who showed that $R_{\mathbb{F}_{2}}(s,d)$ is finite for every choice of $s$ and $d$. The result was subsequently simplified (and generalized) by Spencer in \cite{spencer}, however the quantitative bounds implied by either of these arguments are all rather poor, due to the repeated uses of the Hales-Jewett Theorem. Despite the rich activity on their combinatorial counterparts over the years, there hasn't been much progress on these Ramsey numbers since.

In \cite{NN21}, Nelson and Nomoto established a new connection between $R_{\mathbb{F}_{2}}(2,d)$ and the $\chi$-boundedness of certain classes of binary matroids, which in some sense reignited interest for this problem. Recall that a class $\mathcal{M}$ of matroids is $\chi$-bounded by a function $f$ if $\chi(M) \leq  f(\omega(M))$ for all $M \in \mathcal{M}$, and that $\mathcal{M}$ is $\chi$-bounded if it is $\chi$-bounded by some $f$. Here $\omega(M)$ denotes the {\it{clique number}} of the matroid $M$ and $\chi(M)$ denotes its {\it{chromatic number}}. In particular, for a simple binary $n$-dimensional matroid $M=(E,G)$ defined on a finite projective geometry $G:=\operatorname{PG}(n-1,2)$ and having ground set $E \subset G$, $\omega(M)$ is defined to be the dimension of the largest subgeometry of $M$ contained in $E$, and $\chi(M)$ is defined as $\chi(M)=n - \omega(M^{c})$, where $M^{c} (G\setminus E,G)$ is the matroid complement of $M$. These parameters represent natural analogues of the usual clique number $\omega(G)$ and the chromatic number $\chi(G)$ of a graph $G$, respectively. See for example \cite{Oxley} for a discussion. A simple binary matroid is called {\it{claw-free}} if none of its rank-$3$ flats are independent sets. Unlike claw-free graphs, however, claw-free (simple binary) matroids are not necessarily $\chi$-bounded \cite{claw}: for all $k \geq 0$, there is a so-called even-plane matroid $M$ with $\omega(M)=2$ and $\chi(M) \geq k$. A simple binary matroid $M$ is called an {\it{even plane matroid $M$}} if all of its 3-dimensional induced restrictions have even-sized ground sets. It is not difficult to see that such matroids are claw-free, and so this implies that the class of claw-free matroids is not $\chi$-bounded. One of the main results in \cite{NN21} is that even-plane matroids present the only obstruction to $\chi$-boundedness.

\begin{theorem}[Nelson-Nomoto]
If $N$ is an even-plane matroid, then the class of $N$-free, claw-free matroids is $\chi$-bounded by the function $f(d) =  R_{\mathbb{F}_{2}}(2,d+1) + d(\dim(N)+4)$. 
\end{theorem}

Moreover, in \cite{NN21} Nelson and Nomoto also showed that $R_{\mathbb{F}_{2}}(2,d) \leq (d+1)2^{d}$, thus improving upon the bounds from \cite{GLR72} and \cite{spencer}, and establishing that the class of $N$-free, claw-free matroids above is $\chi$-bounded by an exponential function. In a very recent paper \cite{FY}, Frederickson and Yepremyan further improved this estimate to $R_{\mathbb{F}_{2}}(2,d) \leq 1.566^d$ (by using some new ideas which also allow one to improve upon a few other particular cases of the Graham-Leeb-Rothschild theorem).

On the other hand, it is well-known that the class of claw-free graphs is $\chi$-bounded by a quadratic function: in a claw-free graph $G$, each vertex has the property that its neighborhood has no independent set of size $3$, so the maximum degree of $G$ is upper bounded by the usual off-diagonal Ramsey number $R(3,\omega(G)+1)$; see for example \cite{CS} and the references therein for more details. In light of this analogy, it thus seems natural to ask whether or not it could also be the case that for any even-plane matroid $N$, the class of $N$-free, claw-free matroids is $\chi$-bounded by a function $f(d)$ that grows only polynomially in $d$ (or in other words, is polynomially $\chi$-bounded). 

In this paper, we settle this question in the affirmative by showing that $R_{\mathbb{F}_{2}}(2,d)$ is polynomial in $d$. Specifically, our main result is the following estimate:

\begin{theorem} \label{main}There exists an absolute constant $C>0$, such that for every integer $d\geq 1$:
$$R_{\mathbb{F}_{2}}(2,d) \leq Cd^{7}.$$
\end{theorem}

This result settles a recent question of Nelson \cite{Nelson}, alongside the original problem of Nelson-Nomoto from \cite{NN21} (also reiterated in \cite{FY}), which originally asked to determine whether or not $R_{\mathbb{F}_{2}}(2,d)$ is subexponential.

The proof of Theorem \ref{main} will build upon the original argument from \cite{NN21}, and provide an upper bound for $R_{\F_2}(2,d)$ via a reduction to a well-studied additive combinatorics problem, originally posed by Green in \cite{Green}: given a set $A \subset \mathbb{F}_{2}^{n}$ with density $\alpha \in [0,1]$, what is the largest subspace that we can find in $A+A$? To derive a polynomial bound for $R_{\mathbb{F}_{2}}(2,d)$, we will establish the following new result for this problem, by using ideas from the breakthrough paper of Kelley and Meka on sets of integers without three-term arithmetic progressions \cite{KelleyMeka}. 

\begin{theorem}\label{newSanders}
Let $n \geq 1$. There exists an absolute constant $c>0$ such that the following holds: for any subset $A\subset \F_2^n$ with $|A| = \alpha 2^{n}$, the set $A+A$ must always contain a subspace $V$ of dimension 
$$\operatorname{dim}(V) \geq c\left(\frac{n}{\LL(\alpha)^5}\right)^{1/2}-1.$$
\end{theorem}

In the regime where $1/\alpha$ is large with respect to $n$, Theorem \ref{newSanders} improves upon a well-known result of Sanders \cite{Sanders}. We will discuss this additive combinatorics problem in more detail in Section \ref{overview}, after first revisiting its connection with the Ramsey number $R_{\mathbb{F}_{2}}(2,d)$.

Our argument will also allow us to get strong bounds for the following multi-color generalization of the function $R_{\mathbb{F}_{2}}(2,d)$. To formulate this more general statement, it will be convenient to use the following notation: given a vector space $V$, let ${V \brack k}$ denote as usual the collection of all $k$-dimensional linear subspaces of $V$. Given an $r$-tuple of positive integers $(d_1,\dots,d_r)$, let $R_{\F_2}(d_1,\ldots,d_r)$ denote the smallest $n$, such that for every $r$-coloring $C:{\F_2^n \brack 1} \to \{1,\dots,r\}$, there exists some $i\in \{1,\dots,r\}$ and a subspace $V\le \F_2^n$ of dimension $d_i$ such that $C(W) = i$ for all $1$-dimensional subspaces $W \subset V$. From \cite{GLR72} and \cite{spencer}, we know that $R_{\F_2}(d_1,\ldots,d_r)$ is finite for every choice of $(d_1,\dots,d_r)$. In what follows, we establish the following quantitative improvement when $d_{1}=\ldots=d_{r-1}=2$.

\begin{theorem}\label{multimain}
    Let $r \geq 2$ and let $R_{\F_2}^{(r)}(2;d)$ denote the $r$-color Ramsey number $R_{\mathbb{F}_{2}}(2,\ldots,2,d)$. There exists an absolute constant $C$ such that the following holds: for every integer $d\geq 1$, 
    \[R_{\F_2}^{(r)}(2;d) \le C r (\log r+d)^7.\]
\end{theorem}

This not only recovers Theorem \ref{main} in the case when $r=2$, but it also provides an upper bound of the same quality for all $r \geq 2$ (up to constant factors depending on $r$). We will present the proof of Theorem \ref{multimain} in Section \ref{sect3}, after sketching the argument for two colors in Section \ref{overview}. 
\bigskip

{\bf{Notational conventions}}. We will use standard asymptotic notation, as follows. For functions $f = f(n)$ and $g = g(n)$, we write $f = O(g)$ to mean that there is a constant $C$ such that
$|f(n)| \leq C|g(n)|$ for sufficiently large $n$. Similarly, we write $f = \Omega(g)$ to mean that there is a constant $c > 0$ such that $f(n) \geq c|g(n)|$ for sufficiently large n. Finally, we write $f = \Theta(g)$ to mean that $f=O(g)$ and $g=O(f)$, and we write $f = o(g)$ or $g = \omega(f)$ to mean that
$f(n)/g(n) \to 0$ as $n \to \infty$. We write $O_{H}(1)$ for some unspecified constant that can be chosen as
some bounded value depending only on $H$. Whenever convenient, we will also use the Vinogradov notation $f \ll g$ to denote that $f = O(g)$, that is, there exists some constant $C > 0$ such that $f < Cg$. Similarly, we will write $f \gg g$ whenever $g \ll f$.


Given a finite additive abelian group $G$ and a set $A \subset G$, we write $\mu(A) := |A|/|G|$ for the density of $A$ in $G$. Given two sets $A,X \subset G$, we write $\mu_X(A) := |A\cap X|/|X|$ for the relative density of $A$ inside $X$. We will also use $1_A$ to denote the characteristic function of the set $A$, and use $\mu_A$ for the normalized characteristic function $\frac{1}{\mu(A)}1_A$.

For any two functions $f,g : G \to \mathbb{R}$, we denote by $f\circ g$ the difference convolution of $f$ and $g$, namely
$$f \circ g (x) = \frac{1}{|G|} \sum_{y \in G} f(x+y)g(x).$$
Note that $\mu_A\circ \mu_B(x) := \P(a-b=x)$ where $a\sim A,b\sim B$ are sampled uniformly and independently from their respective sets. 

Last but not least, given an integer $k\ge 1$, we sometimes use the shorthand $[k] := \{1,\dots,k\}$. Given a positive real number $a$, we also write $\LL(a)$ to denote the least integer $m\ge 1$ such that $2^m a\ge 1$.


\bigskip

{\bf{Acknowledgements}}. We would like to thank Bryce Frederickson and Liana Yepremyan for useful discussions. We thank Ryan Alweiss and an anonymous referee for pointing out several typographical errors in the original version of this writeup. 

\section{Proof overview for Theorem \ref{main}
}\label{overview}

In the same spirit as the arguments from \cite{NN21} and \cite{FY}, our starting point will be the observation that in order to obtain an upper bound for $R_{\F_2}(2,d)$, it will suffice to study a certain additive combinatorics problem. 

\begin{obs} \label{sum-free}
    If $R_{\F_2}(2,d)>n$, then there must exist some $S\subset \mathbb{F}_{2}^{n} \setminus \{(0,\ldots,0)\}$ such that $S+S$ and the complement $S^c := \mathbb{F}_{2}^{n}\setminus S$ contain no $d$-dimensional subspaces.
\end{obs}

Indeed, let us denote the additive group $\F_2^n$ by $G$ and the identity vector $(0,\ldots,0)$ by $0_{G}$ for convenience, and then note that $R_{\F_2}(2,d)>n$ implies the existence of a partition of $G\setminus \{0_G\}$ into sets $R$ and $S$ such that:
\begin{itemize}
    \item $R\cup \{0_G\} = G\setminus S$ does not contain a $d$-dimensional subspace;
    \item $S\cup \{0_G\}$ does not contain a $2$-dimensional subspace.
\end{itemize}\noindent The second bullet says that $S$ does not contain any pattern of the form $x,y,x+y$ with $x,y \in \mathbb{F}_{2}^{n}$ and such that $x\neq y$; in other words, $S$ is sum-free. This implies that $S+S\subset G\setminus S$, which by the first bullet means that $S+S$ cannot contain a $d$-subspace.

We next note that the complement of $S$ lacking $d$-dimensional subspaces tells us something about the density of $S$ in $G$. To this end, we recall the following classical result going back to Bose and Burton \cite{Bose}: 
\begin{lemma}\label{complement has subspace}
    Let $G = \F_2^n$ be a vector space, and $S\subset G\setminus \{0_G\}$ have density $\sigma:=\mu(S)$. Consider some $d\in [n]$, and assume $\sigma <1/2^d$ so that $|S| <\frac{2^n-1}{2^d-1}$. Then $S^c$ contains a vector space $V\le G$ with $\dim(V) \ge d$.
    \end{lemma}
   The above result is sharp up to a factor of $\approx 2$. Indeed, if $V'\subset G$ has dimension $n+1-d$, then $\dim(V'\cap V)\ge 1$ for each $V\le G$ with $\dim(V)=d$. Thus, taking $S= V'\setminus\{0_G\}$, we have that $S^c$ contains no subspace with dimension $d$. In fact, \cite[Theorem~2]{Bose} shows that this construction of $S$ is the worst-case scenario; however, for the sake of self-containment, we decided to instead state the slightly weaker version above in order to include a short probabilistic proof.
   
    \begin{proof}
        Sample $\mathbf{V}$ uniformly at random among all the $d$-dimensional subspaces of $G$. For every $s\in G\setminus \{0_G\}$, we shall prove $P_s := \P(s \in \mathbf{V})$ is equal to $(2^d-1)/(2^n-1)$. Assuming this, the result follows by a union bound, since 
        $$\P(\mathbf{V}\cap S\neq \emptyset)\le |S|\frac{2^d-1}{2^n-1}<1.$$ 
        To start, note that for $s,s'\in G\setminus \{0_G\}$, there is an  automorphism $\phi:G\to G$ sending $s$ to $s'$ (by changing bases appropriately). Meanwhile, for any automorphism $\phi$, the distribution of $\phi(\mathbf{V})$ should remain uniform over vector spaces of dimension $d$. Thus, there is some constant $P$ such that $P_s = P$ for all $s\in G\setminus \{0_G\}$. By double-counting, we get
        \[2^d-1 = \E[|\mathbf{V}\setminus \{0_G\}|] = \sum_{s\in G\setminus\{0_G\}} \P(s\in \mathbf{V}) = (2^n-1)P,\]which rearranges to give the result.\end{proof}

Putting together Observation \ref{sum-free} and Lemma \ref{complement has subspace}, we arrive at the following reduction, first noted essentially in \cite{NN21}. 

\begin{obs}\label{right reduction}
    If $R_{\F_2}(2,d)>n$, then there must exist a subset $S\subset \mathbb{F}_{2}^{n}$ where $\mu(S)\ge 2^{-d}$ and such that $S+S$ contains no $d$-dimensional subspaces.
\end{obs}

This naturally leads us to the following additive combinatorics problem, originally posed by Green in \cite{Green}: given a set $A \subset \mathbb{F}_{2}^{n}$ with density $\alpha \in [0,1]$, what is the largest subspace that one can find in $A+A$? 

In the case when $\alpha$ is fixed, the best known quantitative result for this problem comes from \cite{Sanders} (see also \cite{CLS}). 

\begin{theorem}[Sanders]\label{sumset linear} The exists an absolute constant $c>0$ such that the following holds: for any subset $A\subset G:= \F_2^n$, we have that $A+A$ contains a subset $V\le G$ with \[\dim(V)\ge c \mu(A) n.\]
\end{theorem}
Note that Theorem \ref{sumset linear} combined with Observation~\ref{right reduction} immediately tells us that 
$$R_{\F_2}(2,d) \le \frac{d}{c}2^{d}< 2^{(1+o(1))d}.$$


In general, given $n\ge 1$ along with $\alpha>0$, define $f(n,\alpha)$ to be the maximal $d$ such that: for every $A\subset \F_2^n$ with $\mu(A)\ge \alpha$, there exists $d$-subspace $V\le \mathbb{F}_{2}^{n}$ with $V\subset A+A$. We briefly recall what is known here. Thus far in the literature, everyone has primarily focused on the regime where $\alpha>0$ is fixed and then $n$ tends to $\infty$. Using this notation, Theorem~\ref{sumset linear} states that 
\begin{equation} \label{Sanderss}  
f(n,\alpha) \gg \alpha n.
\end{equation}

The best known upper bound for fixed $\alpha \in (0,1/2)$ is
\[f(n,\alpha) \le n-(C_\alpha-o(1)) \sqrt{n},\]where $C_\alpha>0$ is chosen such that if $X\sim \NN(0,1)$ is a standard Gaussian, then $\P(X<-C_\alpha)=\alpha$; this is due to a beautiful ``niveau-set construction'' where $A$ is taken to be set of vectors with Hamming weight at most $n/2-(C_\alpha/2-o(1)) \sqrt{n}$. We refer to \cite[Section~3]{Sanders} for more details. Despite the general belief that this upper bound is closer to the truth (i.e., that $f(n,\alpha)\ge (1-o(1))n$), Theorem~\ref{sumset linear} remains the `state-of-the-art' in the fixed density regime. 

Nevertheless, for the purposes of estimating $R_{\F_2}(2,d)$ it turns out that the density $\alpha$ goes to $0$ very fast as $n$ grows, and in this case our Theorem \ref{newSanders} improves substantially upon the result of Sanders. Using the new notation, Theorem \ref{newSanders} states the following: given $n\ge 1$ and $\alpha>0$, 
\begin{equation} \label{supercool}
f(n,\alpha)+1\gg \left(\frac{n}{\LL(\alpha)^5}\right)^{1/2}.
\end{equation}
Indeed, note that \eqref{supercool} instantly implies that there exists an absolute constant $C>0$ such that $f(Cd^7, 2^{-d})\ge d$, and by Observation \ref{right reduction} this subsequently yields that $R_{\F_2}(2,d) = O(d^{7})$.



\subsection{Proof of Theorem \ref{newSanders}} 

The new lower bound for $f(n,\alpha)$ from \eqref{supercool} can be proven using recent ``approximate Bogolyubov-Ruzsa results'', which have been made possible by the recent breakthrough paper by Kelley and Meka \cite{KelleyMeka}. 


Roughly speaking, an ``{\it{exact}} Bogolyubov-Ruzsa result'' takes a dense subset $A$, and then finds a large subspace $V$ contained in some type of sumset of $A$ (say $A-A$); meanwhile, an ``approximate Bogolyubov-Ruzsa result'' instead finds a large subspace $V$, such that the intersection of $V$ with $A-A$ is at least a $(1-\epsilon)$-fraction of $V$. 

The specific result we need is the following lemma. Results like the one below are known to those familiar with the work of Kelley and Meka \cite{KelleyMeka}, but for completeness we note it follows from a more general result that we prove in the next section (see Remark \ref{justify approx bogo}).
\begin{lemma}\label{approximate bogo}
    Let $G = \F_2^n$ and $A\subset G$ with $\alpha:= \mu(A)>0$. Given $\gamma>0$, we can find $V\le G$ with $\cdim(V)\ll \LL(\alpha)^5\LL(\gamma)^2$, such that writing 
    \[D:= \{v\in V: \mu_A\circ \mu_A(v) \ge \frac{1}{2}\mu(V)^2\}\]to denote the set of ``locally popular diffences'', we have
    \[\mu_V(D)\ge 1-\gamma.\]
\end{lemma}

\begin{remark}
    To prove Theorem~\ref{newSanders}, it would be satisfactory to replace `$D$' by `$A-A$'. However, the proof easily gives this stronger form, so we saw no reason not to mention it.
\end{remark}
\noindent The important property here is that the dependence on the codimension is polylogarithmatically-efficient with respect to both $1/\alpha$ and $1/\gamma$ (which should be optimal, up to the powers of logs).

Given Lemma~\ref{approximate bogo}, along with Proposition~\ref{complement has subspace} from before, we may quickly obtain Theorem~\ref{newSanders}.
\begin{proof}[Proof of Theorem~\ref{newSanders} assuming Lemma~\ref{approximate bogo}]
    Let $C_0$ be the implicit absolute constant from Lemma~\ref{approximate bogo}.

    Then suppose we are given a non-empty subset $A\subset G:= \F_2^n$, and write $\alpha := \mu(A)>0$. Now, let $C$ be some sufficiently large absolute constant. If $n< C\LL(\alpha)^5$, then the conclusion of Theorem~\ref{newSanders} is trivially satisfied by taking $V = \{0_G\}$ which has dimension $0 \ge \frac{1}{C}\left(\frac{n}{\LL(\alpha)^5}\right)^{1/2}-1$.

    Otherwise, by the largeness of $C$, we will have that $\frac{n}{2C_0\LL(\alpha)^5}>\frac{C}{2C_0}>1$. In which case
    we may take $\gamma\in (0,1)$ so that $C_0 \LL(\alpha)^5\LL(\gamma)^2 = n/2$, meaning 
    \[\log_2(1/\gamma) = \left(\frac{n}{2C_0\LL(\alpha)^5}\right)^{1/2}.\]Write $d:= \lfloor \left(\frac{n}{2C_0\LL(\alpha)^5}\right)^{1/2}\rfloor$, so that $\gamma \le 2^{-d}$.
    
    Applying Lemma~\ref{approximate bogo} (which holds with the constant $C_0$) to the set $A$ and the constant $\gamma$, we can find some subspace $V\le G$ with $\dim(V)\ge n/2$, such that defining
    \[D:= \{v\in V:\mu_A\circ \mu_A(v)\ge \frac{1}{2}\mu(V)^2\}\]to be the set of locally popular differences, we have $\mu_V(D)\ge 1-\gamma$. Noting that $D\subset A-A = A+A$ (recall that $A=-A$, as $G=\F_2^n$), it will suffice to show that $D$ contains a $d$-space (recalling $d\gg (n/\LL(\alpha)^5)^{1/2}$).
    
    We shall prove this using Lemma~\ref{complement has subspace}. Let $S:= V\setminus D$, so that $D = S^c$ (where the complement is taken with respect to $V$). Clearly, $\mu_V(S)\le \gamma \le 2^{-d}$; so we will be done if we can confirm that the trivial non-degeneracy assumptions in Lemma~\ref{complement has subspace} are also satisfied.
    
    First, note that $\mu_A\circ\mu_A(0_V) = 1/\alpha \ge 1$, thus $0_V\not \in S$, as desired. Also, without loss of generality, we may suppose $n\ge 4$. Whence, $d \le n/2\le \dim(V)$ (recalling $d\le n^{1/2}$). This establishes the two non-degeneracy assumptions, and thus the result.
\end{proof}

\section{Proof of Theorem~\ref{multimain}} \label{sect3}

A key ingredient of this section is the following density-increment result.

\begin{lemma} [{\cite[Proposition~9]{BSnew}}, in characteristic two]\label{BS increment}
    Let $V\cong \F_2^n$ be a vector space, and consider $A,C\subset V$ with $\mu(A) = \alpha,\mu(C) = \gamma$. Suppose that $|\langle \mu_A\circ \mu_A,\mu_C\rangle -1|>1/2$.

    Then there exists a coset $U = t+V'$ where $\cdim(V')\ll \LL(\alpha)^4\LL(\gamma)^2$ and $\mu_U(A)\ge (1+1/128)\alpha$.
\end{lemma}
\begin{remark}
    In the original work of Kelley and Meka \cite{KelleyMeka},  a slightly weaker version of this result is proven where $\cdim(V)\ll \LL(\alpha)^4\LL(\gamma)^4$. The improved $\gamma$-dependence in the above lemma is based off a combinatorial trick which improves the quantitative performance of an intermediate step of the proof. In the following applications, the same proofs will continue to work with the previous increment bounds, only with slightly worse numbers. For example, we would instead get an upper bound of $R_{\F_2}(2,d)\le O(d^9)$.
\end{remark}



Our main result will be a variant of a ``sparsity-expansion dichotomy'' established by Chapman and Prendiville \cite[Theorem~2.1]{chapman}. 

\begin{proposition}\label{stronger dichotomy}
    Consider subsets $A_1,\dots, A_r\subset G:= \F_2^n$. For each $\alpha,\gamma\in (0,1)$, there exists some subspace $V\le G$ with $\cdim(V) \ll r \LL(\alpha)^5\LL(\gamma)^2$ such that, for each $i\in [r]$, one of the following holds:
    \begin{itemize}
        \item (Sparsity):
        \[\mu_V(A_i) < \alpha;\]
        \item (Expansion):
        \[\mu_V(D_i) > 1-\gamma,\]
        where
        \[D_i:= \{v\in V: \mu_{A_i}\circ \mu_{A_i}(v) \ge \frac{1}{2}\mu(V)^2\}.\]
    \end{itemize}
\end{proposition}
With Proposition~\ref{stronger dichotomy} in hand, it will be easy to deduce our multi-color result. This can be done by setting $\alpha=\gamma:= \frac{1}{10r2^{r}}$, and using the same sorts of reductions as in the last section. We defer the formal details to the end of this section, and first handle the proof of Proposition \ref{stronger dichotomy}. As in \cite{chapman,sanders2}, let us start by first considering the following quantity. 

\begin{definition}
    Given a group $G$ and subsets $S,X\subset G$, denote by
    $$\mu_X^*(S):= \max_{g\in G}\left\{\frac{|(g+X)\cap S|}{|X|}\right\}$$
    the so-called \textit{maximal density} of $S$ within (translates of) $X$.
\end{definition}

Whenever the underlying group $G$ is the additive group $\mathbb{F}_{2}^{n}$, we will actually only consider this quantity when $X$ is a subspace of $\F_2^n$. In this particular situation, we record the following simple observation.
\begin{lemma}\label{mono max}
    Consider subspaces $V' \subset V \subset \mathbb{F}_{2}^{n}$. Then 
    \[\mu_{V'}^*(S)\ge \mu_V^*(S)\]for every $S\subset \mathbb{F}_{2}^{n}$.
    \begin{proof}
        This follows from an averaging argument. For convenience, let us write again $G:=\mathbb{F}_{2}^{n}$. Pick $g_0 \in G$ such that 
        $$\mu_V^*(S) =  \mu_V(g_0+S) = \mu_{V-g_0}(S).$$
        Since $V' \subset V$, we can pick $v_1,\dots,v_\ell\in V$ so that
        \[\bigsqcup_{i=1}^\ell v_i+V' = V.\]Therefore,
        \[\mu_{V-g_0}(S) = \E_{i\in [\ell]}[\mu_{V'+v_i-g_0}(S)].\]It follows that there must be some outcome $i$ where $\mu_{V'+v_i-g_0}(S)\ge \mu_{V-g_0}(S)$, so
        \[\mu_{V'}^*(S) \ge \mu_{V'+v_i-g_0}(S)\ge \mu_{V-g_0}(S) = \mu_V^*(S)\]as desired.
    \end{proof}
\end{lemma}

We can now prove our main iteration lemma.
\begin{proposition}\label{maximal increment}
    Consider subsets $A_1,\dots, A_r\subset G:= \F_2^n$, along with some subspace $V\le G$. Also fix some constants $\alpha,\gamma\in (0,1)$.
    
    Now suppose there was some $i\in [r]$ such that both of the following properties held:
    \begin{itemize}
        \item (failure of sparsity): \[\mu_V^*(A_i)\ge \alpha;\]
        \item (failure of expansion):
        \[\mu_V(D_i) \le 1-\gamma,\]where
        \[D_i := \{v\in V: \mu_{A_i}\circ \mu_{A_i}(v)\ge \frac{1}{2}\mu(V)^2\}.\]
        Then there exists $V'\le V$ with $\cdim(V')-\cdim(V)\ll \LL(\alpha)^4\LL(\gamma)^2$, so that
        \[\mu_{V'}^*(A_i) \ge (1+1/128)\mu_{V}^*(A_i)\](meanwhile $\mu_{V'}^*(A_j)\ge \mu_V^*(A_j)$ for all other $j\in [r]$).
    \end{itemize}
    \begin{proof}
        Throughout we write $A$ to denote $A_i$. Pick $g\in G$ such that $\mu_V^*(A) = \mu_V(A+g)$. Write $A' :=(A +g)\cap V$. By our first assumption, we have $\mu_V(A')\ge \alpha$. 
        
        Next, take 
        \[D' := \{v\in V: \mu_{A'}\circ \mu_{A'}(v)\ge \frac{1}{2}\}.\]Some renormalizing gives the pointwise bound $
        \mu_A\circ \mu_A(x)\ge \mu_{A'}\circ \mu_{A'}(x)\cdot \mu(V)^2$ for $x\in G$. Indeed, we have $\mu_A\circ\mu_A(x)\ge (\mu(A')/\mu(A))^2 \mu_{A'}\circ \mu_{A'}(x)$ (as $A'\subset A$), while $\mu(A')\ge \mu(V)\mu(A)$ follows from unravelling definitions\footnote{Note $\mu(A') = \mu(V)\mu_V(A')$. By definition of $A'$ $\mu(A')=\mu_V^*(A)$. Finally Lemma~\ref{mono max} gives $\mu_V^*(A)\ge \mu_G^*(A) = \mu(A)$.}. This gives the inclusion
        \[D_i = \{v\in V: \mu_A\circ \mu_A(v)\ge \frac{1}{2}\mu(V)^2\}\supset D',\]thus our second assumption implies that $\mu_V(D') \le 1-\gamma$.

        Take $C := V\setminus D'$. We have that $\mu(C) \ge \gamma$. Meanwhile, we have that $\langle \mu_{A'}\circ \mu_{A'},\mu_C\rangle \le 1/2$, whence we can apply Lemma~\ref{BS increment} to find a coset $U=t+V'$ with $V'\le V$  and $\cdim(V')-\cdim(V)\ll \LL(\alpha)^4\LL(\gamma)^2$, so that\[\mu_U(A') \ge (1+1/128)\mu_V(A') =(1+1/128)\mu_V^*(A) .\] Noting that $\mu_{V'}^*(A) \ge \mu_{V'}((g-t)+A) \ge (1+1/128)\mu_V^*(A)$ completes the claim. (For $j\in [r]\setminus\{i\}$, we get $\mu_{V'}^*(A_j)\ge \mu_V^*(A_j)$ by Lemma~\ref{mono max}.) 
    \end{proof}
\end{proposition}

Now our main proposition follows from running an increment argument.
\begin{proof}[Proof of Proposition~\ref{stronger dichotomy}] Fix $A_1,\dots,A_r \subset G$ and $\alpha,\gamma$ as in the prompt. 

Initialize with $V_0:= \F_2^n$. Then at time $t = 0,1,\dots,$ we take \[i_t:= \inf\left\{i\in [r]: \mu_{V_t}^*(A_i) \ge \alpha\text{ and }\mu_{V_t}(\{v\in V_t: \mu_{A_t}\circ \mu_{A_t}(v) \ge \frac{1}{2}\mu(V_t)^2\})\le 1-\gamma\right\};\]if $i_t = \infty$ (i.e., there are no such indices as above), then we halt at time $\tau := t$. Otherwise, we can invoke Proposition~\ref{maximal increment} to find some $V'\le V_t$ where $\cdim(V')-\cdim(V_t) \ll \LL(\alpha)^4\LL(\gamma)^2$, so that $\mu^*_{V'}(A_{i_t})\ge (1+1/128)\mu_{V_t}^*(A_{i_t})$. We define $V_{t+1} := V'$ and continue iterating the above (until we halt). 

Note that this process must halt after $\tau\ll r\LL(\alpha)$ steps, since for each $i\in [r]$, there can be at most $128\LL(\alpha)$ times $t$ where $i_t = i$. Indeed, if $\#(t'<t:i_{t'} = i) = m>0$, we get that \[1\ge \mu_{V_t}(A_i)\ge (1+1/128)^m\alpha \ge 2^{m/128} \alpha.\]

Meanwhile, we have that $\cdim(V_\tau) = \sum_{t=1}^\tau \cdim(V_t)-\cdim(V_{t-1}) \le \tau O(\LL(\alpha)^4\LL(\gamma)^2)$. Thus $\cdim(V_\tau) \ll r \LL(\alpha)^5\LL(\gamma)^2$. This gives the desired result since the sparsity-expansion dichotomy holds for $V_\tau$.
\end{proof}
\begin{remark}\label{justify approx bogo}
    Note that this proof actually locates a subspace $V\le G$ with $\cdim(V)\ll r\LL(\alpha)^5\LL(\gamma)^2$ so that for each $i\in [r]$, either $\mu_V^*(A_i)<\alpha$ or $\mu_V(D_i)>1-\gamma$. The $r=1$ case will imply Lemma~\ref{approximate bogo}, since if $\mu(A)\ge\alpha$, then $\mu_V^*(A)\ge\alpha$ for all subspaces $V$ (causing the latter condition to hold).
\end{remark}

\subsection{Proof of Theorem~\ref{multimain} given Proposition~\ref{stronger dichotomy}}\label{formal deduction}
Let $\alpha = \gamma := \frac{1}{10r2^{d}}$, and note that $L:= \LL(\alpha) = \LL(\gamma) = O(\log(r)+d)$. Let $C_0$ be the implicit constant from Proposition~\ref{stronger dichotomy}, and consider $n := 2C_0 r\LL(\alpha)^5\LL(\gamma)^2= 2C_0rL^7 = O(r(\log(r)+d)^7)$.

    Furthermore, consider $G:= \F_2^n$ any choice of color classes $A_1,\dots,A_r\subset G$ where each $A_i$ is sum-free (in particular, this implies $0\not\in A_i$ for each $i$). By Proposition~\ref{stronger dichotomy} (with our sets $A_1,\dots,A_r$ and parameters $\alpha,\gamma$), we can find some subspace $V\le G$ with $\dim(V)\ge n/2$, so that the desired sparsity-expansion dichotomy holds. Then, for $i\in [r]$, if $\mu_V(A_i)<\alpha$, we set $S_i := A_i$, and otherwise we take \[S_i := \{v\in V: \mu_{A_i}\circ \mu_{A_i}(v) <\frac{1}{2}\mu(V)^2\}.\]By the sparsity-expansion dichotomy, we have that $\mu_V(S_i) \le \frac{1}{10r2^d}$ for each $i\in [r]$; and deterministically we have that $0\not \in S_i$ (since $0\not \in A_i$ by assumption, and if $A_i$ is non-empty, then $\mu_{A_i}\circ \mu_{A_i}(0) = 1/\alpha\ge 1\ge \frac{1}{2}\mu(V)^2$).
    
    Thus, writing 
    \[S:= \bigcup_{i=1}^r S_i,\]we have that $0\not\in S$ and $\mu_V(S)\le \frac{1}{10 2^d}$ (by union bound). Finally, recalling that $\dim(V)\ge n/2\ge d$, we can invoke Lemma~\ref{complement has subspace} to deduce that $S^c:= V\setminus S$ contains a subspace $V'\le V$ with $\dim(V')\ge d$, which is disjoint from $S$. 
    
    For each $i\in [r]$, we have that $S \supset S_i$, which implies $V'$ is disjoint from $A_i$. Indeed, either we have that
    $S_i := A_i$ (which directly implies disjointness), or $S_i \supset V\setminus (A_i-A_i)$ (which implies $V'\subset A_i-A_i$, meaning $V'$ should be disjoint from $A_i$ by sum-freeness). So, it follows that $V'\subset G\setminus \bigcup_{i=1}^r A_i$, which implies the claim.

\section{Another application}

In this section, we would like to note that our methods can be extended further to provide an even stronger form of Theorem~\ref{multimain} (namely, Theorem~\ref{solutionfree strong bounds}), which gives similar results for arbitrary finite-fields. 

Given a prime $p \geq 2$, let us say that a set $S\subset \F_p^n$ is \textit{solution-free}, if there exists some choice $a,b\in \{1,\dots,p-1\}$ such that there is no $x,y,z\in S$ satisfying $a\cdot (x-y) = b\cdot z$. Note that every sum-free set $A\subset \F_p^n$ is a solution-free set (by taking $a=b=1$, and rearranging $(x-y)=z$ as $x=y+z$). 

This definition is motivated by ``homogeneous partition-regular equations'' (discussed below). Obtaining quantitative bounds for coloring problems related to these equations is an important topic in additive combinatorics, as such questions represent generalizations of the celebrated theorem of Schur (which states that for any finite coloring of $\Z_{>0}$, there must be a color class which is not sum-free). We will establish nearly sharp bounds in the finite-field setting for such problems.   

In the integer setting, it is known that for any fixed $a,b\in \Z_+$, and any $r\ge 1$, there exists some finite $N$ so that for any $r$-coloring $C:[N]\to [r]$, we can find some monochromatic triple $x,y,z$ such that $a(x-y) = bz$. Specializing to the case of $a=b=1$, this is the content of Schur's Theorem (which alternatively says that $[N]$ cannot be covered by $r$ sum-free sets (assuming $N$ is large enough with respect to $r$)). In \cite{cwalina}, Cwalina and Schoen studied the growth of $N= N_{a,b}(r)$ for any fixed $a,b\in \Z_+$ as $r\to \infty$. They demonstrated that \[\exp(\Omega_{a,b}(r)) \le  N_{a,b}(r) \le \exp(O_{a,b}(r^4\log(r)^4)).\]We defer the interested reader to \cite{cwalina} for further motivation on the study of such quantities; here, we will provide some nearly sharp bounds to a finite-field analogue of this problem.

\begin{theorem}\label{solutionfree strong bounds}
    For each prime $p$, there exists a constant $C_p$ such that the following holds:

    Consider solution-free subsets $A_1,\dots,A_r\subset G:= \F_p^n$. Furthermore, suppose that $n>C_pr(\log(r)+d)^7$ for some $d\ge 1$.
    
    Then there exists $V\le G$ with $\dim(V)\ge d$ so that $V$ is disjoint from $\bigcup_{i=1}^r A_i$.
\end{theorem}

    This in particular tells us that $\F_p^n \setminus \{0\}$ cannot be covered by $\ll n^{1-c}$ sum-free sets for any absolute constant $c> 0$ (recalling that sum-free sets are solution-free). By considering using dyadic intervals to cover each of the $n$ coordinates, it is not hard to see that $O(\log(p)n)$ sum-free sets suffice to cover $\F_p^n$, thus this is nearly sharp. It is conjectured that $\Omega(\log(p) n)$ sets are also required, due to connections with conjectures regarding the asymptotic growth of the $r$-color Ramsey $R(3,\ldots,3)$. 
    
    Previously, the best-known additive proof\footnote{Due to connections with the $r$-color Ramsey number of triangles, it is known that you can't use at most $o( n/\log(n))$ sum-free sets to cover $\F_p^n$. But this proof is purely graph theoretical, and cannot be generalized to give bounds in the case of solution-free sets, or even to sets which all lack solutions to the equation $x-y = 2z$.} could only prove that $n^{1/3}$ sum-free sets are not sufficient to cover $\F_2^n\setminus \{0\}$ \cite[Section~2]{chapman}. For more general solution free-sets, the aforementioned upper bound of Cwalina-Schoen in the integer setting morally says that $n^{1/4-c}$ solution-free sets should not be enough to cover $\F_p^n\setminus \{0\}$ for any absolute constant $c>0$ \cite[Theorem~1.7]{cwalina}.

The starting point (for proving Theorem~\ref{solutionfree strong bounds}) is \cite[Proposition 9]{BSnew}.
\begin{lemma}[Bloom-Sisask]
    Fix a prime $p$. Let $V \cong \F_p^n$. Consider $A,C \subset V$ with $\mu(A)= \alpha,\mu(C) = \gamma$. Suppose that $|\langle \mu_A\circ \mu_A, \mu_C\rangle -1|>1/2$.

    Then there exists a coset $U= t+V'$ where $\cdim(V')\ll_p \LL(\alpha)^4\LL(\gamma)^2$ and $\mu_U(A)\ge (1+1/128)\alpha$.
\end{lemma}

Repeating the arguments from the last section, it is not difficult to establish the following more general result.

\begin{proposition}\label{general dichotomy}
    Fix a prime $p$.

     Consider subsets $A_1,\dots, A_r\subset G:= \F_p^n$. For each $\alpha,\gamma\in (0,1)$, there exists some subspace $V\le G$ with $\cdim(V) \ll_p r \LL(\alpha)^5\LL(\gamma)^2$ such that, for each $i\in [r]$, one of the following holds:
    \begin{itemize}
        \item (Sparsity):
        \[\mu_V(A_i) < \alpha;\]
        \item (Expansion):
        \[\mu_V(D_i) > 1-\gamma,\]
        where
        \[D_i:= \{v\in V: \mu_{A_i}\circ \mu_{A_i}(v) \ge \frac{1}{2}\mu(V)^2\}.\]
    \end{itemize}
\end{proposition}

We will leave the details of establishing Proposition~\ref{general dichotomy} to the interested reader. We now include the main ideas for how to prove Theorem~\ref{solutionfree strong bounds} (at a high level).

The rough idea is that if $A_i$ is a solution-free set, with no solutions to $a\cdot (x-y) = b\cdot z$, then defining $A_i' := a\cdot A_i$, we have that $0\not\in A_i'$, and for any subspace $V\le G$, $\mu_V(A_i') = \mu_V(A_i)$ and $V\subset A_i'-A_i'$ implies $V\cap A_i = \emptyset$. Meanwhile, similarly to Lemma~\ref{complement has subspace}, if $S\subset V\cong \F_p^n$ is a set with $0\not\in S$ and $\mu_V(S) < p^{-d}$ for some $d\le n$, then we have that there exists a subspace $V'\subset  V\setminus S$ with $\dim(V')\ge d$. One can now hope to argue in the same way as in Subsection~\ref{formal deduction} (now taking $\alpha = \gamma:= \frac{1}{10r p^d}$). We outline the argument below.

\begin{proof}[Proof sketch of Theorem~\ref{solutionfree strong bounds}]Consider $r,d\ge 1$, and set $\alpha= \gamma := \frac{1}{10 r p^d}$. Note that $\LL(\alpha) = O_p((\log(r)+d))$.

Assume that $n>  C_p r\LL(\alpha)^5\LL(\gamma)^2$ for some sufficiently large constant $C_p$ (which is allowed to depend on $p$); note that the lower bound on the right hand side is $O_p(r(\log(r)+d)^7)$.

Now consider any solution-free sets $A_1,\dots,A_r\subset \F_p^n$. For each $i$, we can find some $a,b\in \{1,\dots,p-1\}$, so that there is no $x,y,z\in A_i$ satisfying $a\cdot (x-y)=b\cdot z$; set $A_i' := a\cdot A_i$.

Using Proposition~\ref{general dichotomy}, we can find $V\le \F_p^n$ with $\dim(V)\ge n/2 \ge d$, and a partition $[r] = I_1\sqcup I_2$, so that:
\begin{itemize}
    \item $\mu_V(A_i') \le\alpha$ for $i\in I_1$;
    \item $\mu_V(V\setminus (A_i'-A_i'))\le \gamma$ for $i\in I_2$.
\end{itemize}For $i\in I_1$, set $S_i := A_i'$, and for $i\in I_2$, set $S_i := V\setminus (A_i'-A_i')$. As discussed above, since $A_i$ is solution-free, if $V'\le V$ is disjoint from $S_i$, then $V'$ is disjoint from $A_i$.

Now, taking $S:= \bigcup_{i=1}^r S_i$, we get that $\mu_V(S) \le \sum_{i=1}^r \mu_V(S_i) \le \frac{1}{10rp^d}<p^{-d}$. Furthermore, it is easy to show that $0\not\in S$. Thus (using the analog of Lemma~\ref{complement has subspace} which we discussed previously) there exists some $V'\le V$ with $\dim(V')\ge d$ which is disjoint from $S$, proving the theorem.
\end{proof}

\section{Concluding remarks and open problems}

In this paper, we proved that $R_{\mathbb{F}_{2}}(2,d) =O(d^{7})$, and more generally that
$$R_{\F_2}^{(r)}(2;d)=O_{r}(d^{7})$$
holds for all $r \geq 2$. We would like to begin this section by contrasting our multi-color result with the following result of Alon and R\"odl from \cite{AlonRodl}, generalizing the works of Ajtai-Komlo\'s-Szemer\'edi \cite{AKS} and Kim \cite{Kim} about the usual off-diagonal Ramsey number $R(3,m)$. 

\begin{theorem}[Alon-R\"odl]\label{AR}
Let $R^{(r)}(3;m)$ denote the $r$-color Ramsey number $R(3,\ldots,3;m)$. For every $r \geq 2$, we have that
$$R^{(r)}(3;m) = \Theta(m^{r}/ \operatorname{poly log} m).$$
\end{theorem}
For convenience, $\operatorname{poly log} m$ here stands for an unspecified polynomial factor in $\log m$. For $r=2$, $R^{(2)}(3;m)=R(3,m)$, in which case Theorem \ref{AR} recovers (up to this polylogarithmic factor) the fact that $R(3,m) = \Theta(m^2 / \log m)$. For every $r \geq 3$, Theorem \ref{AR} implies for instance that 
$$\lim_{m \to \infty} \frac{R^{(r)}(3;m)}{R^{(r-1)}(3;m)} = \infty,$$
formerly a conjecture of Erd\H{o}s and S\'os \cite{ErdosSos} already in the case $r=3$. It would be tempting to make the analogous conjecture for $R_{\mathbb{F}_{2}}^{(r)}(2;d)$, namely that for every $r \geq 3$, 
\begin{equation} \label{ARno}
\lim_{d \to \infty} \frac{R_{\F_2}^{(r)}(2;d)}{R_{\F_2}^{(r-1)}(2;d)} = \infty
\end{equation}
must always hold. However, despite the successful analogies between $R_{\mathbb{F}_{2}}(2,d)$ and $R(3,m)$ thus far, we suspect \eqref{ARno} should not hold. Instead, we would like to make the following competing conjecture regarding the growth of $R_{\F_2}(2,d)$. 



\begin{conj} \label{linear}
    $$R_{\F_2}(2,d) = \Theta(d).$$
\end{conj}

A standard probabilistic construction can be used to see that\footnote{Technically speaking, the ``all-red coloring'' of $\F_2^{d-1}$ already proves that $R_{\F_2}(2,d)\ge d-1$. However, using a random coloring and a straightforward application of the Lov\'asz Local Lemma, one can show that $R_{\F_2}(2,d) \ge (2-o(1))d$ (thus we have $R_{\F_2}(2,d)-d \gg d$). We note that we are not aware of any better bounds than this basic probabilistic argument. In fact, we do not even know how to prove a better bound using (finitely many) additional colors (for example, it would already be interesting to show that $R^{(1000)}(2;d)\ge (2+c-o(1))d$ for some absolute constant $c>0$).} $R_{\F_2}(2,d) = \Omega(d)$, so in order to prove Conjecture \ref{linear} it would suffice to show that $R_{\F_2}(2,d) = O(d)$. Furthermore, note that 
$$R_{\F_2}^{(r)}(2;d) \leq R_{\F_2}^{(r-1)}(2;R_{\F_2}(2,d))$$ holds for every $r \geq 3$. Therefore, if $R_{\F_2}(2,d) = O(d)$, then $R_{\F_2}^{(r)}(2;d)=O_r(d)$ follows for every $r \geq 3$ as well. Moreover, it is clear that $R_{\F_2}^{(r)}(2,d) \geq R_{\F_2}^{(r-1)}(2,d)$ holds for every $r \geq 3$, thus Conjecture \ref{linear} implies that $R_{\F_2}^{(r)}(2;d) = \Theta_{r}(d)$ holds for every $r \geq 2$. 

That being said, we still believe that it would be interesting to distinguish quantitatively between, say, $R_{\F_2}(2,d)$ and $R_{\F_2}(2,2,d)$. For example, if \eqref{ARno} is too much to ask for, can one at least show that 
$$\lim_{d \to \infty}\left(R_{\F_2}^{(r)}(2;d) - R_{\F_2}^{(r-1)}(2;d)\right) = \infty$$
for every $r \geq 3$? 

In light of the additive combinatorics approach developed in this paper, an important problem that is directly relevant to the faith of Conjecture \ref{linear} seems to be the question of determining whether or not there exists some absolute constant $C>0$ such that $f(Cd,2^{-d})\ge d$ for all $d\ge 1$. 


There are also several natural open problems that one may pose about the general behavior of $R_{\mathbb{F}_{p}}(s,d)$ for other choices of $p,s \geq 2$ (as $d$ goes to $\infty$). We would like to share a few of our favorite ones. 

\begin{conj} \label{polynomial} For every $s \geq 2$, there exists some $C=C(s)>0$ such that
$$R_{\mathbb{F}_{2}}(s,d)=O(d^{C}).$$
\end{conj}

In other words, we believe that $R_{\mathbb{F}_{2}}(s,d)$ is polynomial in $d$ for every $s \geq 2$. We suggest a possible approach, which could potentially extend the ideas of this paper. Given $n,k\ge 1$ and $\alpha >0$, define $f_k(n,\alpha)$ to be the largest $d$, such that given any $A\subset G:= \F_2^n$ with $\mu(A)\ge \alpha$, we have that there exist some $V\le G$ with $\dim(V) \ge d$, such that for each $v_1,\dots,v_k\in V$, there exists $a\in A\setminus V$ such that the parallelepiped
\[\{a+\sum_{i\in I} v_i: I\subset [k]\}\]is contained entirely inside $A$. Note that $f_1(n,\alpha)$ is equal to the function $f(n,\alpha)$ we introduced previously.

Using arguments like before, one can show that if \[f_{s-1}(n,2^{-d})\ge R_{\F_2}(s-1,d),\] then we have that $R_{\F_2}(s,d)\le n$. Indeed, either the first color class will be ``sparse'' (density $<2^{-d}$), which implies a dimension $d$ subspace in the complement, or one should be able to pass to some subspace $V$ where every $(s-1)$-dimensional $V'$ subspace has some disjoint `coset'/parallelepiped which can extend $V'$ to a $s$-subspace (thus $V$ should be colored to not have a red $(s-1)$-subspace nor blue $d$-subspace, allowing us to iterate). Thus, if for each $k\ge 1$, one could show that $f_k(Cd^C,2^{-d})\ge d $ for some constant $C=C_k$, then $R_{\F_2}(s,d) \ll d^{O_s(1)}$ for every $s \geq 2$. Furthermore, if $f_k(C_k d,2^{-d}) \ge d$ holds for some constant $C_k$, the sketch above would readily imply that $R_{\F_2}(s,d)=O_s(d)$ for every $s \geq 2$. 

Last but not least, it would also be interesting to get reasonable bounds for the off-diagonal numbers in different finite fields. For example, it is already not clear how to extend the arguments of this paper to show that $R_{\F_3}(2,d)$ is polynomial in $d$. This presents new challenges, as now there is a difference between the set of $1$-dimensional subspaces of $\F_3^n$, and the set of non-zero points in this space (indeed, one can define a 2-coloring $C:\F_3^n\setminus \{0\}\to \{1,2\}$ so that $C(v)\neq C(2\cdot v)$ for all non-zero $v$, in which case every non-trivial subspace of $\F_3^n$ receives multiple colors). We currently do not have ideas for any additive combinatorial result that would readily give polynomial bounds in this setting.

\end{document}